\documentclass[11pt,twoside]{article}

\usepackage{latexsym,amssymb,amsthm,amsmath,amscd}

\theoremstyle{plain}

\newtheorem{theorem}{Theorem}[section]

\newtheorem{lemma}{Lemma} [section]

\newtheorem{proposition}{Proposition}[section]

\newtheorem{corollary}{Corollary}[section]

\theoremstyle{remark}

\theoremstyle{definition}
\def\<{\left < }
\def\>{\right >}
\def\({\left ( }
\def\){\right )}
\def\r{\eqref }

\def \na {\nabla}
\def\vp{\varphi}

\newtheorem{defn}{Definition}[section]

\numberwithin{equation}{section}

\begin{document}

\vbox{\hbox{\small Bull. Transilvania Univ. Brasov $\bullet$ Vol 1(50) - 2008}
\hbox{\small Series III: Mathematics, Informatics, Physics, 59--78}}

\vskip 1.5truecm


\centerline{\large{\bf Submanifolds  of warped product  manifolds
$I\times_f S^{m-1}(k)$}}

\centerline{\large{\bf  from a $p$-harmonic viewpoint}}

\medskip
\centerline{\bf Bang-Yen Chen and Shihshu Walter Wei$^*$}
\medskip
\begin{abstract} {\small We study $p$-harmonic maps, $p$-harmonic morphisms, biharmonic
maps, and quasiregular mappings into submanifolds of warped
product Riemannian manifolds ${I}\times_f S^{m-1}(k)\, $ of an
open interval and a complete simply-connecteded
$(m-1)$-dimensional Riemannian manifold of constant sectional
curvature $k$. We establish an existence theorem for $p$-harmonic
maps and give a classification of complete stable minimal
surfaces in certain three dimensional warped product Riemannian
manifolds ${\bf R}\times_f S^{2}(k)\, ,$ building on our previous
work. When $f \equiv\, $ Const. and $k=0$, we recapture a
generalized Bernstein Theorem and hence the Classical Bernstein
Theorem in $R^3$.  We then extend the classification to parabolic
stable minimal hypersurfaces in higher dimensions.
 \smallskip

}
 \smallskip
 \end{abstract}


\footnote[0]{${}^*$ Research was
partially supported by NSF Award No DMS-0508661.}

\section{Introduction}

In \cite {CW2}, we make the first general study of submanifolds in
{\it warped product Riemannian manifolds} $I\times_f S^{m-1}(k)$
of nonconstant curvature from differential geometric viewpoint.
Here $I \subset {\bf R}\, $ is an open interval, $S^{m-1}(k)$ is
an $(m-1)$-dimensional complete, simply-connected, Riemannian
manifold of constant sectional curvature $k$, and $f\, $ is a
warping function.  This is in contrast to the study of
submanifolds in (real, complex, Sasakian, ..., etc.) space forms
due to the simplicity of their curvature tensors (see, for
instance, \cite{book73,book81,c2000}). The study is also in
contrast to the recent work in treating Riemannian warped product
manifolds as submanifolds from the viewpoint of isometric
immersions (cf. \cite{c2002,CW}).

  The
purpose of this paper is to study submanifolds in {\it warped
product Riemannian manifolds} $R^m(k,f) := I\times_f S^{m-1}(k)\,
$ (or simply denoted by $I \times_f S)\, $ of nonconstant
curvature from a $p$-harmonic viewpoint.
 In particular, we study $p$-harmonic maps, $p$-harmonic
morphisms, biharmonic maps, and quasiregular mappings into
submanifolds of $R^m(k,f)\, .$  Furthermore, building on our
previous work, for concave warping function $f$ with bounded
derivative $ |f'| \le {\sqrt{k}}$ on $R$, we give a classification
Theorem of complete stable minimal surfaces in three dimensional
warped product Riemannian manifolds ${\bf R}\times_f S^{2}(k)\, .$
When $f \equiv Const.$ and $k=0$, we recapture a generalized
Bernstein Theorem (\cite {FS, DP, P}, cf. Theorem \ref{T:7.2}) and
hence The Classical Bernstein Theorem (\cite {B}, cf. Theorem
\ref{T:7.3}) in $R^3$. The techniques that we utilized, are
sufficiently general to extend the classification theorem for
surfaces to parabolic stable minimal hypersurfaces in higher
dimensional warped product Riemannian manifolds (cf. Theorem
\ref{T:3.4}).

This article is organized as follows: After this first
introductory section, we recall some necessary formulas, notations
and basic results on warped product manifolds $R^m(k,f)\, ,$ and
our previous work on submanifolds of $R^m(k,f)\, $ in section 2,
and then describe $p$-harmonic maps and $p$-harmonic morphisms
into submanifolds of $R^m(k,f)\, $ in section 3,  biharmonic maps
into $R^m(k,f)$ or submanifolds of $R^m(k,f)$ in section 4,
quasiregular mappings into ${\bf R} \times_{f} S$ in section 5, a
link to manifolds with warped cylindrical ends in section 6,  a
classification theorem of complete stable minimal surfaces in
three dimensional warped product Riemannian manifolds ${\bf
R}\times_f S^{2}(k)\, $ in section 7, a classification of
parabolic stable minimal hypersurfaces in ${\bf R}\times_f
S^{n}(k)\, $ in section 8, and $p$-hyperbolic manifolds and stable
minimal hypersurfaces in ${\bf R}\times_f S^{n}(k)\, $ in section
9.
\section{Preliminaries}

We recall some basic facts, notations,  definitions, and
inequalities for Riemannian submanifolds and warped product
manifolds (cf. \cite{book73,O}), and some known results about
submanifolds of $R^m(k,f)\, $ and ${\bf R} \times_{f} S$ (see
\cite{CW2} for details).

\subsection{Basic equations and inequalities}
Let $M$ be a submanifold of dimension $n\geq 2$ in a Riemannian
manifold $\tilde M$ with Levi-Civita connection $\tilde \nabla$.
Denote by $\nabla$ and $\Gamma (TM)$, the Levi-Civita connection
of $M$ and the (infinite dimensional ) vector space of smooth
sections of a smooth tangent bundle $TM$  of $M$ respectively. The
formulas of Gauss is given by (cf. \cite{book73})
\begin{equation}\label{2.1}\tilde \nabla_XY=\nabla_XY+h(X,Y),\end{equation} for   $X,Y \in \Gamma (TM)\, ,$ where $h$ is the second fundamental
form of $M$ in $\tilde M\, .$ The mean curvature vector field of a
submanifold $M$ is defined by $H=\frac{1}{n}{\rm trace}\, h$. A
submanifold $M$ in $\tilde M$ is called \emph{totally geodesic}
(resp. \emph{minimal}) if  $h\equiv 0$ (resp.  $H\equiv 0$).   A
minimal hypersurface $M$ in a Riemannian manifold $\tilde M$  is
said to be \emph{stable minimal}, if it is a local minimal of area
functional. Thus, if $M$ is stable minimal in $\tilde M$, then for
every $\phi \in C_0^{\infty}(M)\, ,$
\begin{equation}\label{2.20}\int _M (Ric^{\tilde M} (\nu) + |A|^2 )\phi^2 dv \le \int _M |\nabla
\phi|^2 dv, \end{equation}
 where $\nu$ is a unit normal
vector field to $\tilde M$, $|A|^2$ is the squared of the length
of the second fundamental form of $M$ in $\tilde M$.
\subsection{Warped products} Let $B$ and $F$ be two Riemannian manifolds  of positive dimensions equipped with Riemannian metrics $g_B$ and $g_F$, respectively, and let $f$ be a positive function on $B$. Consider the product manifold $B\times F$ with its projection $\pi:B\times F\to B$ and $\eta:B\times F\to F$. The {\it warped product} $\tilde{M}=B\times_f F$ is the manifold $B\times F$ equipped with the Riemannian structure such that
\begin{equation}\label{E:warped} ||X||^2=||\pi_*(X)||^2+f^2(\pi(x)) ||\eta_*(X)||^2\end{equation}  for any tangent vector $X\in T_x \tilde{M}$. Thus, we have $g=\pi ^*(g_B)+(f \circ \pi)^2\eta^*(
g_F)$. The function $f$ is called the {\it warping function\/}  of
the warped product.

  Let ${\mathcal L}(B)$ and ${\mathcal L}(F)$ denote the set of lifts of vector fields on $B$ and $F$ to $B\times_f F$ respectively. For each $q\in F$, the horizontal leaf $\eta^{-1}(q)$ is a totally geodesic submanifold of $B\times_f F$ isometric to $B$.  For each $p\in B$, $\pi^{-1}(p)$ is an $(n-b)$-dimensional totally umbilical submanifold of $B\times_f F$ that is homothetically isomorphic to $F$ with scalar factor $\frac 1{f(p)}$, where $b$ is the dimension of $B$. The submanifolds $\pi^{-1}(p) =\{ p\} \times F, p\in B,$ and $\eta^{-1}(q)= B \times \{q\}, q \in F$ are called fibers and leaves respectively. A vector field on $\tilde{M}$ is called vertical if it is always tangent to fibers; and horizontal if it is always orthogonal to fibers. We use the corresponding terminology for individual tangent vectors as well. A vector field on $\tilde{M}$ is called {\it basic} if $X$ is horizontal and $\pi$-related to a vector field $X_*$ on $B$.

 Let  ${\mathcal H }$ and ${\mathcal V}$ denote the projections of tangent spaces of $\,\tilde{M}\,$ onto the subspaces of horizontal and vertical vectors, respectively. We use the same letters to denote the horizontal and vertical distributions.

On the warped product ${\bf R}\times_f S\, ,$  let $t$ be an
arclength parameter of ${\bf R}$. Let us denote by $\partial_t\,
,$ the lift to ${\bf R}\times_f S$ of the standard vector filed
$\frac {d}{dt}$ on ${\bf R}$. Thus, $\partial_t \in {\mathcal
L}({\bf R})\, .$

For each vector field $V$ on ${\bf R}\times _f S$, we decompose
$V$ into a sum
 \begin{equation}\label{2.9}V=\varphi_V \partial_t+\hat V,\end{equation} where $\varphi_V=\<V,\partial_t\>$ and $\hat V$ is the vertical component of $V$ that is perpendicular to $\partial_t\, ,$ or the projection of $V_{(p,q)}$ onto its vertical subspace $T_{(p,q)} ( p \times S)$.

\begin{lemma} \label{L:2.2} The curvature tensor $\tilde R$ of  $R^m(k,f)$ satisfies
\begin{equation}\begin{aligned}\label{2.10}& \tilde R(\partial_t,X)\partial_t
=\frac{f''}{f}X,\;\; \\&\tilde
R(X,\partial_t)Y=\<X,Y\>\frac{f''}{f}\partial_t,\;\; \\& \tilde
R(X,Y)\partial_t=0,
\\& \tilde R(X,Y)Z=\frac{k-f'{}^2}{f^2}\{\<Y,Z\>X-\<X,Z\>Y\} \end{aligned}\end{equation} for $X,Y,Z \in {\mathcal L}(S)$.
\end{lemma}
\begin{proof}
 Follows from \cite[p. 210]{O}.
\end{proof}

\subsection{Notions of transverse and $\mathcal H$-submanifolds}
 By a {\it slice} of  ${\bf R} \times_f S$ we mean a
hypersurface of $R^m(k,f)$ given by  $S(t_0):=\{t_0\}\times S$ for
some $t_0\in R$.    A submanifold $M$ of ${\bf R}\times_f S$ is
called a {\it transverse submanifold}  if it is contained in a
{\it slice} $S(t_0):=\{t_0\}\times S$ (with the metric:
$f^2(t_0)g_k)$ for some $t_0\in R\, ,$ i.e. $\partial_t^T=0\, ,$
where $\partial_t^T\, $ is the tangential projection of
$\partial_t\, $ onto $M\, .$

For simplicity, we call a submanifold $M$ in ${\bf R}\times_f S$
an $\mathcal H$-{\it submanifold} if the horizontal vector field
$\partial_t$ is tangent to $M$ at each point on $M\, ,$ i.e.
$\partial_t^\perp=0\, ,$ where $\partial_t^\perp$ denotes the
normal component of $\partial_t$ to $M$.

\subsection{Submanifolds of ${\bf R}\times_f S$}

Let $M$ an $n$-dimensional submanifold of $R^m(k,f)$ and
$e_1,\ldots,e_n$ an orthonormal frame on $M$. Then
\begin{equation}\begin{array}{rll}\label{3.11}&\Phi=\sum_{j=1}^n\varphi_j^2\end{array}\end{equation}
is called the {\it total scalar projection} of $TM$ onto
$\partial_t$, where $\varphi_j = \langle e_j,
\partial _t\rangle$.

For a submanifold $M$ in $R^m(k,f)$, the total scalar projection
$\Phi$ satisfies $0\leq \Phi\leq 1,$ with $\Phi=0$ (respectively,
$\Phi=1$) holding at each point if and only if $M$ is a transverse
submanifold (respectively,  $M$ is an $\mathcal H$-submanifold).

\subsection{Minimal submanifolds of ${\bf R}\times_f S$}

\begin{proposition} {\rm \cite{CW2}} \label{P:2.1} Let $M$ be a minimal submanifold of ${\bf R}\times_f S$.

$(a)$ If $(\ln f)''+k/f^2=0$ on $R$, then the Ricci curvature of
$M$ satisfies \begin{equation}\begin{array}{rll}\label{5.1}&
Ric(X)\leq \frac{(n-1)(k-f'{}^2)}{f^2}\<X,X\>,
\end{array}\end{equation}
for $ \,  X \in \Gamma (TM)$. The equality sign of (\ref{5.1})
holds identically if and only if $M$ is a totally geodesic
submanifold of constant curvature $(k-f'{}^2)/f^2$.

$(b)$ If $\, (\ln f)''+k/f^2\ne 0$ on $R$, then the Ricci
curvature of $M$ satisfies
\begin{equation}\begin{array}{rll}\label{5.2}& Ric(X)\leq
\left\{((2-n)\varphi_X^2-\Phi)\(\frac{k}{f^2}+(\ln f)''\)
+\frac{(n-1)(k-f'{}^2)}{f^2}\right\}\<X,X\>.
\end{array}\end{equation}

The equality sign of (\ref{5.2}) holds identically if and only if
one of the following two cases occurs:

$(b.1)$ $M$ a transverse submanifold which lies in a slice
$S(t_0)$ with $f'(t_0)=0$ as a totally geodesic submanifold;

$(b.2)$  $M$ is an $\mathcal H$-submanifold which is locally the
warped product $ I\times_{ f} N^{n-1}$ of $I$ and a totally
geodesic submanifold   $N^{n-1}$ of  $S$.

Furthermore, if  case $(b.1)$ occurs, then $Ric(X) =
\frac{(n-1)(k-f'{}^2)}{f^2}\<X,X\>, X \in TM.$
 \end{proposition}

 \begin{proposition}  {\rm \cite{CW2}} \label{P:5.2}
Let $M$ be a Riemannian $n$-manifold whose scalar curvature
satisfies
\begin{equation}\begin{array}{rll}\label{5.4}\tau\geq  \frac{n(n-1)(\tilde
k-f'{}^2)}{2f^2}\end{array}\end{equation} for some real number
$\tilde k\geq k$, where $f$ be a positive function satisfying
$f'\ne 0$ and $(\ln f)''> - k/f^2$ on $ I$. Then $M$ cannot be
isometrically immersed in $R^m(k,f)$ as a minimal submanifold.
\end{proposition}

\begin{proposition}  {\rm \cite{CW2}} \label{P:5.3} Let $M$ be a
Riemannian $n$-manifold whose scalar curvature satisfies
\begin{equation}\begin{array}{rll}\label{5.6}\tau >  \frac{n(n-1)( {k}-f'{}^2)}{2f^2}\end{array}\end{equation} at one point, where $f$ is a positive function satisfying $(\ln f)''\ge - k/f^2$ on $ I$.
Then   $M$ cannot be isometrically minimally immersed in
$R^m(k,f)$.
\end{proposition}

\subsection{Classification of parallel submanifolds in $R^m(k,f)$}

\begin{theorem}   {\rm \cite{CW2}}  \label{T:6.1} If $R^m(k,f)$ contains no open subsets of constant curvature, then  a submanifold $M$ of $R^m(k,f)$ is a parallel submanifold if and only if one of the following statements holds:

\noindent  $(1)$  $M$  is a transverse submanifold which lies in a
slice $S(t_0)$ of $R^m(k,f)$ as a parallel submanifold.

\noindent  $(2)$  $M$ is an $\mathcal H$-submanifold which is
locally the warped product $ I\times_{ f} N^{n-1}$, where
$N^{n-1}$  is a submanifold of  $S$. Furthermore, we have

$(2.1)$ if $f'\ne 0$ on $I$, then $M$ is totally geodesic in
$R^m(k,f)$;

$(2.2)$ if $f'= 0$ on $I$, then $N^{n-1}$ is a parallel
submanifold of  $S$.

\end{theorem}

\section{$p$-harmonic maps and $p$-harmonic morphisms into submanifolds in $R^m(k,f)$}

A smooth map $u: M \to N$ is said to be \emph{$p$-harmonic}, $p >
1\, ,$ if it is a critical point of the $p$-energy functional:
$${E}_p(u) = \frac 1p \int _M |du|^p dv$$ with respect to
any compactly supported smooth variation, where $|du|$ is the
Hilbert-Schmidt norm of the differential $du$ of $u$. It follows
from the first variational formula for the $p$-energy functional,
$u$ is $p$-harmonic if and only if $u$ is a weak solution of the
Euler-Lagrange equation $div (|du|^{p-2}du) =0\, .$ Examples of
$p$-harmonic maps include geodesics, minimal submanifolds,
conformal maps between manifolds of the same dimensions, and
harmonic maps (when $p=2\, $) (cf. \cite{NS,Th,W,W3}). A $C^2$ map
$u: N_1 \to M$ is called a \emph{$p$-harmonic morphism} if for any
$p$-harmonic function $f$ defined on an open set $V$ of $M$, the
composition $f \circ u$ is $p$-harmonic on $u^{-1}(V)$. Examples
of $p$-harmonic morphisms include the Hopf fibrations.

\subsection{Existence of $p$-Harmonic Maps}

\begin{theorem} \label{T:9.0}
Let $M$ be a complete Riemannian manifold and $N$ be a compact
Riemannian manifold with the universal cover $\widetilde N = R
\times_{ f} S^{m-1}(k)\, ,$ where $f$ is a positive convex
function on $R\, ,$ and $k \le 0\, .$ Then any continuous map from
$M$ into $N$ of finite $p$-energy can be deformed to a
$C^{1,\alpha}$ $p$-harmonic map minimizing $p$-energy in its
homotopy class, where $1<p<\infty$.
\end{theorem}

\begin{proof} In view of Lemma \ref{L:2.2}, $\widetilde N$ is a
complete simply-connected Riemannian manifold of nonpositive
curvature. Then the assertion follows from \cite [Theorem 2.1 or
Corollary 2.4]{W4}.
\end{proof}

\subsection{ Maps of compact manifolds}

\begin{theorem} \label{T:9.1}
Every $p$-harmonic map from a compact manifold into $M$ or
$R^m(k,f)$ is constant, provided  $k \le 0\, ,$ $f$ is a positive
convex function on an open interval $I\, ,$ and $M$ is a totally
geodesic submanifold of $\, R^m(k,f)\, .$  In particular, if $M$
is a non-transversal parallel submanifold of $R^m(k,f)$ where $k
\le 0\, ,$ and $f$ is a positive convex function satisfying $f'\ne
0$ and $(\ln f)''\ne {-k}/{f^{2}}$ on an open interval $I\, ,$
then every $p$-harmonic map from a compact manifold into $M$ is
constant. \end{theorem}

\begin{proof}
By Lemma  \ref{L:2.2}, $ I\times_{ f} S$ is nonpositively curved.
Thus, the Gauss curvature equation implies that $Sec^M =
Sec^{R^m(k,f)}\le 0\, .$  It follows that the image under any
$p$-harmonic map of a compact manifold lies in a domain of a
strictly convex function, e.g. the squared distance function (cf
\cite{BO}). But this is impossible unless it is constant by \cite
{WY} or \cite [Theorem 8.1]{W}. This proves the first assertion.
Now the second assertion follows from the Classification Theorem
\ref{T:6.1} of parallel submanifolds, that $M$ has to be totally
geodesic in $R^m(k,f)\, .$
\end{proof}

\begin{corollary} \label{C:9.1} Let $M$ be a submanifold of $R^m(k,f)\, $ as in Theorem \ref{T:9.1}.
  Then there are no compact geodesics (without boundary), and no compact minimal submanifolds in $M$ or in $R^m(k,f)$. \end{corollary}

\begin{proof}  This follows from the previous Theorem \ref{T:9.1},
and \cite[Theorems 1.10 (i) and 1.14 (i) p.635,637]{W} which state
that a curve parametrized proportionally to the arc length is
$p$-harmonic for any $p \ge 1$ if and only if it is a geodesic,
and an isometric immersion of $M$ is minimal if and only if it is
$p$-harmonic for every $1 < p <\infty$.
\end{proof}

\subsection{ Maps of complete noncompact manifolds}

In the following, we assume that $N_1$ is a complete noncompact
Riemannian manifold  and  $B(x_0;r)$ is  the geodesic ball of
radius $r$ centered at $x_0 \in N_1$.
We recall some notions from \cite{WLW}:

\begin{defn}\label{D:3.1}
A  function $f$  on $N_1$ is said to have \emph{$p$-{finite
growth}} $($or, simply, \emph{is $p$-{finite}}$)$ if there exists
$x_0 \in N_1$ such that
\begin{equation} \lim_{r \to \infty} \inf\frac{1}{r^p}\int_{B(x_0;r)} |f|^{q}dv < \infty ;   \label{9.1} \end{equation}
it has \emph{$p$-{infinite growth}} $($or, simply, \emph{is
$p$-infinite}$)$ otherwise. \label{def4.1}\end{defn}

\begin{defn}\label{D:3.2}
A function $f$ has \emph{$p$-mild growth} $($or, simply, \emph{is
$p$-mild}$)$ if there exist  $ x_0 \in N_1\, ,$ and a strictly
increasing sequence of $\{r_j\}^\infty_0$ going to infinity, such
that for every $l_0>0$, we have
\begin{equation}\sum\limits_{j=\ell_0}^{\infty}\text{\small $\bigg(\frac{(r_{j+1}-r_j)^p}{\int_{B(x_0;r_{j+1})\backslash B(x_0;r_{j})}|f|^qdv}\bigg)$}^{\frac1{p-1}}=\infty \,;   \label{9.2}\end{equation}
and has \emph{$p$-severe growth} $($or, simply, \emph{is
$p$-severe}$)$ otherwise. \label{def4.2}\end{defn}

\begin{defn}\label{D:3.3}
A function $f$ has \emph{$p$-obtuse growth} $($or, simply,
\emph{is $p$-obtuse}$)$ if there exists $x_0 \in N_1$ such that
for every $a>0$, we have
\begin{equation}\int^\infty_a \text{\small $\bigg( \frac{1}{\int_{\partial B(x_0;r)}|f|^qdv}\bigg)$}^\frac{1}{p-1}dr = \infty \, ;   \label{9.3}\end{equation}
and has \emph{$p$-acute growth} $($or, simply, \emph{is
$p$-acute}$)$ otherwise . \label{def4.3}
\end{defn}

\begin{defn}\label{D:3.4}
A function $f$ has \emph{$p$-moderate growth} $($or, simply,
\emph{is $p$-moderate}$)$ if there exist  $ x_0 \in N_1$, and
$F(r)\in {\mathcal F}$, such that \begin{equation} \lim \sup _{r
\to \infty}\frac {1}{r^p F^{p-1} (r)}\int_{B(x_0;r)} |f|^{q}dv <
\infty . \label{9.4}
\end{equation}
And it has \emph{$p$-immoderate growth} $($or, simply, \emph{is
$p$-immoderate}$)$ otherwise,  where
\begin{equation} {\mathcal F} = \{F:[a,\infty)\longrightarrow
(0,\infty) |\int^{\infty}_{a}\text{\small$
\frac{dr}{rF(r)}$}=+\infty \ \ for \ \ some \ \ a \ge 0 \}\, .
\label{9.5}\end{equation} {\rm (Notice that the functions in
{$\mathcal F$} are not necessarily monotone.)} \label{def4.4}
\end{defn}

\begin{defn}\label{D:3.5}
A function $f$ has \emph{$p$-small growth} $($or, simply, \emph{is
$p$-small}$)$ if there exists $ x_0 \in N_1\, ,$ such that for
every $a>0\, ,$ we have \begin{equation}\int^\infty_a\text{\small
$\bigg( \frac{r}{\int_{ B(x_0;r)}|f|^qdv}\bigg)$}^\frac{1}{p-1}dr
= \infty ; \label{9.6}
\end{equation}
and  has \emph{$p$-large growth} $($or, simply, \emph{is
$p$-large}$)$ otherwise.\label{def4.5}
\end{defn}

We introduce the following notion in \cite {W3}:
\begin{defn}\label{D:3.6}A function $f$ has \emph{$p$-balanced growth} $(or, simply,
\emph{is $p$-balanced})$ if $f$ has one of the following:
\emph{$p$-finite}, \emph{$p$-mild}, \emph{$p$-obtuse},
\emph{$p$-moderate}, or \emph{$p$-small} growth, and has
\emph{$p$-imbalanced growth} $(or, simply, \emph{is
$p$-imbalanced})$ otherwise.
\end{defn}

\noindent The above definitions \ref{D:3.1}-\ref{D:3.6} depend on
$q$, and $q$ will be specified in the context in which the
definition is used.
\smallskip

\begin{theorem} \label{T:9.3}
Let $M$ and $R^m(k,f)$ be as in  Theorem \ref{T:9.1}, and $u$ from
$N_1$ into $M$ or $R^m(k,f)$ be a smooth (a) harmonic map where
$p=2$ or (b) $p$-harmonic morphism with $p
>2$. Then $u$ is either a constant or the function $dist^{2}(u(x),y_0)$
on $N_1$ has \emph{$p$-imbalanced} growth for all $q > p-1$. Here
$dist^{2}(u(x),y_0)$ is the square of the distance between $u(x)$
and a fixed point $y_0$ in $M$ or $R^m(k,f)\, .$
\end{theorem}

\begin{proof}
This follows immediately from \cite [Theorem 5.4.(i).]{WLW}, based
on a composition formula and estimates on $p$-subharmonic
functions\cite [Theorems 2.1-2.5]{WLW}. For $p > 2\, ,$ the
composition $dist^{2}(u(x),y_0)$(of a $p$-harmonic morphism $x
\longmapsto u(x)$ and a convex function $y \longmapsto
dist^{2}(y,y_0)$) is $p$-subharmonic on $N_1$(cf.\cite [Theorem
5.2.]{WLW}). For $p=2$, see also \cite {CW}.
\end{proof}

\vskip.0in
\section{Biharmonic maps into $R^m(k,f)$ or submanifolds of $R^m(k,f)$}

We apply our  results from previous sections to study biharmonic,
conformal, $p$-harmonic maps into $R^m(k,f)$ or into submanifolds
of $R^m(k,f)\, .$ We also study isometric minimal immersions in
$R^m(k,f)$ and its submanifolds.
\smallskip

 A smooth map $u : M \to N$ between
two Riemannian manifolds is said to be \emph{ biharmonic} if $u$
is a critical point of bi-energy: $$E_{(2)}(u) = \int _M |(d +
d^\ast)^2 u|^2 dx = \int _M |\tau (u) |^2 dx $$ with respect to
any compactly supported variation, and \emph{ polyharmonic of
order $k$} if $u$ is a critical point of $$E_{(k)}(u) = \int _M
|(d + d^\ast )^k u|^2 dx $$ with respect to any compactly
supported variation, where $\tau (u)$ is the tension field of $u$,
and $d^\ast$ is the adjoint of the exterior differential operator
$d$.

A $C^{\infty}$ section of a bundle over a Riemannian manifold
 has \emph{$p$-imbalanced} growth if its norm is so, and \emph{$p$-balanced} growth otherwise.

\begin{theorem}\label{T:10.1} Let $M$ be a submanifold of $R^m(k,f)\, $ as in  Theorem \ref{T:9.1}.
Let  $u$ be a smooth biharmonic isometric immersion of a complete
manifold $N_1$ into $M$ or $R^m(k,f)$.  If for some $q>2$, $\tau
(u)$ has \emph{$2$-balanced} growth, then we have:
\begin{itemize}

 \item [(1)]   $\tau (u)$ is parallel. Further,  if $u$ is not harmonic, then $N_1$ is parabolic.

\item [(2)]   $u$ is either the unique harmonic map unless it is a
constant or maps $N_1$ onto a closed geodesic $\gamma$ in $M$ or
$R^m(k,f)$ $($in the latter case, we have uniqueness up to
rotations of $\gamma\,)$; or $u$ is of rank one in which case
$\tau (u)$ at each point is tangent to the image curve of $u$.

\item [(3)]  If, for some $x_0 \in M$ and $q=2$, $du$ has
\emph{$2$-balanced} growth, then $u$ is a harmonic map minimizing
energy in its homotopy class.

\item [(4)]  Under the assumption of $(3)$, if for some $q>2$, and
$y_0 \in u(N_1)$, its \emph{distance function}, defined by $
dist(u(x),y_0)$ for $x \in N_1\, ,$ has \emph{$2$-balanced}
growth, then $u$ is constant.
 \end{itemize}
\end{theorem}

 \begin{proof}
We note that both $M$ and $R^m(k,f)$ are simply-connected
manifolds of nonpositive curvature, but are not necessarily
complete. However, the conclusions follow by proceeding exactly as
in the proof of \cite [Theorem 9.1]{WLW}(cf. \cite {WLW2}).
  \end{proof}

\begin{corollary} \label{C:10.1} Let $M$ and $R^m(k,f)\, $ be as in Theorem \ref{T:10.1}.
Let  $u$ be a smooth biharmonic isometric immersion of a complete
manifold $N_1$ into $R^m(k,f) $($resp. $\, M$\, $$)$ with mean
curvature vector field $H$ which has \emph{$2$-balanced} growth.
Then we have:

\begin{itemize}
 \item [(i)] $N_1$ is a minimal submanifold of $R^m(k,f)$ $($resp. of  $M$ $)$ for $u : N_1 \to R^m(k,f)$ $($resp. $u : N_1 \to M)$.

 \item [(ii)] $u$ is $p$-harmonic for every $1 < p <\infty \, .$

  \item [(iii)] $u$ is polyharmonic of order $j$ for every $j \in \{1,2, \cdots \}$.
 \end{itemize}
\end{corollary}

\begin{proof}  We first assume the case $u(N_1) \subset M\, ,$ and note that if $u$ is an isometric immersion, then its tension field $\tau (u)$ agrees with its mean curvature $H$ and $|du|= \sqrt {\dim
N_1}\, .$ It follows from Theorem \ref{T:10.1} (1) that either $H
\equiv 0\, ,$ then we have proved (i), or $|H| \equiv C$, a
nonzero constant, i.e. $u$ is not harmonic. But if $|H| \equiv C
\ne 0$, then the growth assumption on $H$ implies the same growth
condition on $|du|= \sqrt {\dim N_1}\,.$ Thus, by Theorem
\ref{T:10.1} (3), $u$ would be harmonic and hence $H \equiv 0$, a
contradiction. This proves that $N_1$ is a minimal submanifold of
$M$,  and hence a minimal submanifold of $R^m(k,f)$ by Theorem
\ref{T:6.1} that $M$ is a totally geodesic submanifold of
$R^m(k,f)$. The same technique also proves the case $u(N_1)
\subset R^m(k,f)\, ,$ and the assertion (i) follows. Now
assertions (ii) and (iii) follow from \cite[Theorem 1.14,
p.637]{W}.
\end{proof}

\begin{theorem}\label{T:10.2}
Let $M$ and $f$ be as in Proposition \ref{P:5.2} or \ref{P:5.3}.
Then there are neither $p$-harmonic  conformal immersions $u : M
\to R^m(k,f)\, , p=\dim N_1\, ,$ nor $p$-harmonic isometric
immersions $u : M \to R^m(k,f)\, ,p
> 1\, .$
\end{theorem}
\begin{proof}
Suppose on the contrary, then by a Theorem of Takeuchi \cite {Th}
or  \cite [Theorem 1.14] {W}, $u(M)$ would be minimal in
$R^m(k,f)$, contradicting Proposition \ref{P:5.2} or \ref{P:5.3}.
\end{proof}

\section{Quasiregular mappings into ${\bf R} \times_{f} S$}

 Our previous ideas can be naturally applied
to the study of quasiregular mappings. These mappings are
generalizations of complex analytic functions on the plane, to
higher dimensional Euclidean spaces; even more generally to
Riemannian $n$-manifolds. While analytic functions pull back
harmonic (resp. superharmonic) functions on an open subset of
${\bf R}^2$ to harmonic (resp. superharmonic) functions,
quasiregular mappings pull back $n$-harmonic (resp.
$n$-superharmonic) functions on manifolds to
$\mathcal{A}$-harmonic (resp. $\mathcal{A}$-superharmonic)
functions (of type $n$) (cf \cite {GLM} and \cite {H}).
\medskip

 We denote by ${\sc W}_{loc}^{1,p}(M)$ the Sobolev space whose real-valued
functions on $M$ are locally $p$-integrable and have locally
$p$-integrable partial distributional first derivatives. A
continuous mapping $u : M \to N$ between two Riemannian
$n$-manifolds is said to be \emph{quasiregular} if $u$ is in
$W_{loc}^{1,n}(M) ,$ and there exists a constant $1 \le \mathbf{K}
< \infty$ such that the differential $du_x$ and the Jacobian
$J_{u}(x)$ satisfy
\begin{equation} |du_x|^n \le \mathbf{K} J_{u}(x)\label{11.1}\end{equation}
for (a.e.) almost every $x\in M ,$ where  the operator norm of
differential
$$|du_x| = \max  \{ du_x ( \xi) : \xi \in T_x(M) , |\xi|=1 \}\,  .$$  A
quasiregular mapping is said to be \emph{quasiconformal} if it is
a homemorphism. A continuous mapping $u : M \to N$ is said to be a
\emph{quasi-isometry} if $u$ is in $W_{loc}^{1,1}(M) ,$ if
$J_{u}(x) \ge 0\, $ a.e., and if there exists a constant $1 \le
\mathbf{L} < \infty$ such that the differential $du_x$ satisfies
\begin{equation} \frac {1}{\mathbf{L}}|\xi| \le |du_x (\xi )| \le \mathbf{L} |\xi|\label{11.2}\end{equation}
for (a.e.) almost every $x\in M ,$ and $\xi \in T_x(M)\, .$
Examples of quasiregular mapping include isometries,
quasi-isometries (with $\mathbf{K}$ = $\mathbf{L}^{2(n-1)}$),
M\"obius maps, and holomorphic maps from the complex plane to a
Riemann surface.

 We denote by ${\mathcal A}$ a
measurable cross section in the bundle whose fiber at a.e. $x$ in
$M$ is a continuous map ${\mathcal A}_x$ on the tangent space
$T_x(M)$ into $T_x(M)$. We assume further that there are constants
$1 < p < \infty$ and $0 < \alpha \le \beta < \infty$ such that for
a.e. $x$ in $M$ and all $h \in T_x(M)\, ,$ we have
\begin{equation}
\langle {\mathcal A}_x(h),h\rangle _M \ge \alpha |h|^p,
\label{11.3}
\end{equation}
\begin{equation}
|{\mathcal A}_x(h)| \le \beta |h|^{p-1} ,  \label{11.4}
\end{equation}
\begin{equation}
\langle {\mathcal A}_x(h_1)-{\mathcal A}_x(h_2),h_1-h_2\rangle_M
> 0,\quad h_1\ne h_2,\label{11.5}
\end{equation}
and
\begin{equation}
{\mathcal A}_x(\lambda h)\equiv  |\lambda|^{p-2}\lambda {\mathcal
A}_x(h),\quad \lambda \in R \smallsetminus \{0\}.   \label{11.6}
\end{equation}

\noindent A function $f\in W_{loc}^{1,p}(M)$ is a weak solution
(resp. supersolution, subsolution) of the equation
\begin{equation}
{\rm div} {\mathcal A}_x(\nabla f) = 0\; \;( resp.\;  \le 0, \;\;
\ge 0),\label{11.7}
\end{equation}
\noindent if for all nonnegative $\varphi \in C_0^{\infty}(M)$
\begin{equation}
\int _M \langle {\mathcal A}_x(\nabla f), \nabla \varphi\rangle dv
= 0\;\; ( resp.\; \ge 0, \; \; \le 0)  \label{11.8}
\end{equation}

The equation \r{11.7} is called ${\mathcal A}$--$ {\it harmonic}$
equation, and continuous solutions of ({\ref{11.7}}) are called
${\mathcal A}$--$ {\it harmonic}$ (of type $p$).  In the case
${\mathcal A}_x(h)\equiv |h|^{p-2}h$, ${\mathcal A}$--harmonic
functions are $p$--harmonic. A lower (resp. upper) semicontinuous
function $ f : M \to R\, \cup \,\{\infty\}$ (resp. $\{-\infty\}
\,\cup\, R$) is ${\mathcal A}$--${\it superharmonic }$ (resp.
${\mathcal A}$-subharmonic) (of type $p$) if it is not identically
infinite, and it satisfies the ${\mathcal A}$--comparison
principle: i.e., for each domain $D \subset  M$ and for each
function $g \in C(\overline D)$ which is ${\mathcal A}$--harmonic
in $D$, $g \le f$ (resp. $g \ge f$) in $\partial D$ implies $g \le
f$ (resp. $g \ge f$) in $D$. An ${\mathcal A}$--superharmonic
(resp. ${\mathcal A}$--subharmonic) function $f$ is called
$p$-superharmonic (resp. $p$-subharmonic) if ${\mathcal
A}_x(h)\equiv |h|^{p-2}h\, .$

 ${\mathcal A}$--superharmonic and ${\mathcal A}$-subharmonic functions are closely related to
 subsolutions and supersolutions of (\ref{11.7}).  For a discussion of the ${\mathcal A}$-harmonic equation, we refer the reader to J. Heinonen, T. Kipel\"ainen and O.
 Martio's book (\cite {HKM}).\medskip

A complete noncompact manifold $M$ is said to be $\mathcal
A$--parabolic (resp. $p$--parabolic) (of type $p$) if every
nonnegative measurable $\mathcal A$-superharmonic (resp.
$p$-superharmonic) (of type $p$) function is constant, and
$\mathcal A$--hyperbolic (resp. $p$--hyperbolic) (of type $p$)
otherwise.

Throughout this section, we assume $S$ is  a Riemannian
$(m-1)$-manifold of constant sectional curvature $k$. We begin
with a general
\medskip

\begin{theorem}\label{T:11.0} Let $N_1$  be an $\mathcal{A}_1$-parabolic manifold  (of type $m$)
and $N_2$  be an $\mathcal{A}_2$-hyperbolic manifold  (of type
$m$). Then there does not exist any quasiregular mapping $u$ from
$N_1$ into $N_2\, ,$ unless it is a constant.
\end{theorem}

The case $\mathcal{A}_1 (\nabla \varphi)= \mathcal{A}_2 (\nabla
\varphi)= |\nabla \varphi|^{m-2}\nabla \varphi\, $ is due to T.
Coulhon, I. Holopainen and L. Saloff-Coste \cite {CHS}:

\begin{proposition} \label{P:11.0}
Every quasiregular mapping $u$ from an $m$-parabolic manifold into
an $m$-hyperbolic manifold is constant.
\end{proposition}

Theorem \ref{T:11.0} recaptures classical {\it Picard's Theorem},
which states that every analytic function $u$ on the complex plane
$\mathbf{C}$ omits at least two different values must be constant.
This is the case for its lift $\tilde{u}: \mathbf{C} \to
\mathbf{D}\, ,$ where $m=2\, , \mathbf{K}=1\, $ in \r{11.1},
${A_1}_x(h) = {A_2}_x(h) =h, N_1 = \mathbf{C}\, $ is parabolic,
and $N_2\, = \mathbf{D}\, ,$ the universal cover of $\mathbf{C}
\backslash \{z_1, z_2,\cdots \}\, ,$ is hyperbolic (cf. also
\cite{W3}).

\begin{proof}
Suppose on the contrary. Let $f$ be a nonconstant positive
$\mathcal{A}_2$-superharmonic function  (of type $m$) on $N_2\, ,$
and $f_j $ be a nonconstant supersolution of
$\mathcal{A}_2$-harmonic equations, where $f_j = \min \{f, j\}$
and $j$ is a positive integer (cf. \cite {HKM}, 7.2,7.20). Then
$u$ would pull back $f_j$ on $N_2$ to a nonconstant positive
supersolution $f_j \circ u$ of $\mathcal{A}_3$-harmonic equations
on $N_1\, ,$ where $\mathcal{A}_3$ is the pull-back of
$\mathcal{A}_2$ under $f_j \circ u$ satisfying \r{11.3}-\r{11.6}
(cf. \cite {H},(2.9a),(2.9b)). It follows that there would exist a
compact set $C \subset N_1$ such that $\inf_{\varphi} \int _{N_1}
|\nabla \varphi|^{m}  dV
> 0\, ,$ where the
infimum is taken over all  $\varphi \ge 1 $ on $C$ and $\varphi
\in C_0^{\infty} (N_1)\,$ (cf. \cite {H},5.2). In view of
\r{11.3}, we would have
\begin{equation}\begin{array}{rll} \inf _{\varphi} \int _{N_1} \< \mathcal{A}_1 (\nabla \varphi),
\nabla \varphi \> dV & \ge \inf _{\varphi} \alpha \int _{N_1}
|\nabla \varphi|^{m}  dV \\&
> 0\end{array}\end{equation}
where the infimum is taken over all  $\varphi \ge 1 $ on $C$ and
$\varphi \in C_0^{\infty} (N_1)\, .$ By an exhaustion argument
(cf. e.g. \cite {WLW},5), based on Harnack's principle, H\"older
continuity estimates, and Arzela-Ascoli Theorem, there would exist
a nonconstant positive $\mathcal{A}_1$-superharmonic function (of
type $m$) on $N_1\,  $( called \emph {Green function} on $M$ for
the operator $\mathcal{A}_1$ ) (cf. \cite {H}, 3.27),
contradicting the hypothesis that $N_1$ is an
$\mathcal{A}_1$-parabolic manifold
 (of type $m$).
\end{proof}

\begin{corollary} \label{C:11.0.1}
Every $m$-harmonic morphism $u$ from an $m$-parabolic manifold
into an $m$-hyperbolic manifold is constant.
\end{corollary}

\begin{proof}
It follows from the fact that every $m$-harmonic morphism $u$
between $m$-manifolds is conformal (cf. \cite {OW}), and hence
quasiregular (in which $\mathbf{K}=1$ in \r{11.1}).
\end{proof}

\begin{theorem}\label{T:11.1} Let $f$ be a positive function on the Euclidean line ${\bf R}$
satisfying $\min \{\frac{f''}{f}, \frac{f'{}^2-k}{f^2}\} \ge a^2$
with $ a > 0 \, .$ Then there does not exist any quasiregular
mapping $u$ from any $\mathcal{A}$-parabolic manifold $N$ (of type
$m$) into ${\bf R} \times_f S,$ unless it is a constant.
\end{theorem}

\begin{proof}  By virtue
of Lemma \ref{L:2.2}, ${\bf R} \times_f S$ is a complete
simply-connected manifold with sectional curvature $\overline K
\le -a^2\, .$ Then, for any domain $\Omega$ relatively compact in
${\bf R} \times_f S$ with smooth boundary $\partial \Omega \, ,$
$x_0 \in \Omega $ and $r(x) = dist (x_0, x)$, we have via Gauss'
lemma, Stokes' Theorem and the Hessian Comparison Theorem that
\begin{equation}\begin{aligned} & \int _{\partial \Omega} 1
\, dS \ge \int _{\partial \Omega} \< \nabla r, \nu \> dS = \int _{
\Omega} \Delta r dV \\&\hskip.3in  \ge (m-1) \int _{
\Omega}\text{\small $ \frac {\cosh ar}{\frac 1a \sinh ar} $} dV
\ge (m-1)\int _{ \Omega} a
dV.\end{aligned}\label{11.8}\end{equation} Hence,
\begin{equation} {\rm Area} (\partial \Omega) \ge (m-1)a \,
{\rm vol}(\Omega).\label{11.9}\end{equation} Now set
\[\Psi(t) = \inf\big \{  {\rm Area} (\partial \Omega):  \Omega \subset
\subset R \times_f S \, ,  \partial \Omega \in C^{\infty}  \, ,
{\rm vol} (\Omega) \ge t \big\}. \] Then, for any $p \in (1,
\infty)\, ,$ $\Psi(t)$ satisfies
\begin{equation} \int _{t_0}^{\infty} \frac {1}{\Psi(t)^{\frac {p}{p-1}}} dt < \infty \label{11.10}\end{equation}
where $t_0
> 0$ is a constant.  It follows from a Theorem of Troyanov \cite {Tr} that
there exists a nonconstant positive supersolution of $p$-harmonic
equation defined on ${\bf R} \times_f S\, ,$ or ${\bf R} \times_f
S\, $ is $p$-hyperbolic for every $p
> 1\, .$  Choose $p = m$, and the assertion follows from Theorem
\ref{T:11.0} (in which $\mathcal{A}_1 = \mathcal{A}\, ,$ and
$\mathcal{A}_2 (h)\equiv |h|^{m-2}h \, $).
\end{proof}

\begin{corollary} \label{C:11.1} Let $f$ and $N$ be as in Theorem \ref{T:11.1}, and $M$ be a
totally geodesic $n$-submanifold of $\, R^m(k,f)\, .$ Then (i)
Every quasiregular mapping $u$ from $N$ into  $\, R^m(k,f)\, $ is
constant.   (ii)
Every quasiregular mapping  from an $\mathcal{A}$--parabolic
$n$-manifold (of type $n$)  into $M\, $ is constant.
\end{corollary}

\begin{proof}
(i) In view of (\ref{11.1}), $u : N \to R^m(k,f)\, $ is
quasiregular as a mapping into $ R \times_f S\, ,$ and hence a
constant by Theorem \ref{T:11.1}. (ii)
By the totally geodesic assumption, $M$ is an $n$-manifold with
sectional curvature bounded above by $-a^2\, ,$ and hence $M$ is
$n$-hyperbolic. The assertion follows from Theorem \ref{T:11.0}.
\end{proof}

\begin{corollary} \label{C:11.2} Let $f$ be as in Theorem \ref{T:11.1}.
If $N$ is a Riemannian $n$-manifold of nonnegative Ricci
curvature, then there does not exist any quasiregular mapping $u$
from $N$  into  ${\bf R} \times_f S,$ unless it is a constant.
\end{corollary}

\begin{proof}
By virtue of Bishop's Volume Comparison Theorem and
$\mathcal{A}$--superharmonic estimates,  $N$ is
$\mathcal{A}$--parabolic (cf. \cite [Corollary 3.3]{WLW}). The
assertion follows from Theorem \ref{T:11.1}.
\end{proof}

\begin{corollary} \label{C:11.3} Let $f$ be as in Theorem \ref{T:11.1}, and let $f_1$ be a positive concave function on the Euclidean line ${\bf R}$
satisfying $f_1^{\prime}\, ^2 \le k\, . $ Then there is neither
any nonconstant quasiregular mapping $u$ from ${\bf R}
\times_{f_{1}} S\, $ into ${\bf R} \times_f S\, ,$ nor nonconstant
quasiregular mapping $u_1 : M_1 \to M$ between complete totally
geodesics $n$-submanifolds $M_1 ( \subset {\bf R} \times_{f_{1}} S
)\, $ and $ M ( \subset {\bf R} \times_{f} S )\,  .$
\end{corollary}

\begin{proof}
By assumption and Lemma \ref{L:2.2}, ${\bf R} \times_{f_{1}} S\, $
has nonnegative sectional curvature. The assertion follows from
Corollary \ref{C:11.2}.
\end{proof}

As noted above, quasiconformal mappings and quasi-isometries are
special cases of quasiregular mappings (in which $\mathbf{K} =
\mathbf{L}^{2(n-1)}$), we have

\begin{corollary} \label{C:11.4} Let $f$ be as in Theorem \ref{T:11.1}.
Then every quasiregular map from ${\mathbb E}^m$ into  ${\bf R}
\times_f S$ is constant.  In particular, there is neither a {\bf
quasi-isometry}  from ${\mathbb E}^m$ into ${\bf R} \times_f S$
whose Jacobian is positive almost everywhere, nor a {\bf
quasiconformal} map from ${\mathbb E}^m$ into ${\bf R} \times_f
S$.
\end{corollary}

\section{A link to manifolds with warped cylindrical ends}

Our previous study can also be linked to manifolds with warped
cylindrical ends.  A manifold $N_1$ is said to \emph {have a
warped cylindrical end} if there exists a compact domain $D
\subset  N_1$ and a compact Riemannian manifold $(K, g_K)$ such
that $N_1\backslash D = (1, \infty) \times _{f_1} K\, ,$ the
warped product of $(1, \infty)$ and $K\, .$ An obvious example is
the Euclidean plane with warping function $f_1(t) = t\, .$ As a
second example, the warped product $I \times _f S\, ,$ where $I =
(1, \infty)$ and $S = S^{m-1}(1)$ is an $(m-1)$-manifold with a
warped cylindrical end, in which $D$ is the empty set, and $f_1 =
f$.

\begin{theorem}\label{T:12.1} Let $N$ be as in Theorem \ref{T:11.1}, and $N_2$ be an $m$-manifold with a warped cylindrical end such that
the warping function $f_2$  satisfies $\int _1^{\infty}\frac
1{f_2(t)} dt < \infty\, ,$ then there does not exist any
nonconstant quasiregular
 mapping $u$ from $N$ into $N_2\, .$  In particular, there is no
 nonconstant $m$-harmonic morphism from $N$ into $N_2\, .$
\end{theorem}

\begin{proof}  According to a Theroem of M. Troyanov \cite {Tr}, an $m$-manifold $N_2$ with a warped cylindrical end is $p$-parabolic
if and only if its warping function $f_2$ satisfies $\int
_1^{\infty}f_2(t)^{\frac {m-1}{1-p}} dt = \infty\, . $ Hence,
$N_2$ is $m$-hyperbolic, and $N$ is $m$-parabolic. Now the
assertion follows from Theorem \ref{T:11.1}, and Corollary
\ref{C:11.0.1}.
\end{proof}

Similarly, we have the following Liouville-type results for
$p$-harmonic morphisms between manifolds with warped cylindrical
ends:

\begin{theorem} \label{T:12.2}  Let $N_i \, (i = 1, 2)$ be an $m_i$-manifold with a warped cylindrical end such that the warping functions
$f_1$ and $f_2$ satisfy $\int_1^{\infty}f_1(t)^{\frac
{m_1-1}{1-p}} dt = \infty\,  $ and $\int _1^{\infty}f_2(t)^{\frac
{m_2-1}{1-p}} dt < \infty\, . $ Then every $p$-harmonic morphism
from $N_1$ into $N_2\,$  is constant.
\end{theorem}

As an obvious example of Theorem \ref{T:12.2}, there does not
exist a nonconstant $p$-harmonic morphism from the Euclidean space
$\mathbb  E^{m_1}$ into $\mathbb E^{m_2}$ for $m_1 \le p < m_2\,.$

In view of the above second example of an $m$-manifold with a
warped cylindrical end,  Theorems \ref{T:11.0}, \ref{T:12.1} and
\ref{T:12.2} yield the following two results.

\begin{corollary}\label{C:12.1} Let $N_i\, $ $(i = 1, 2)\, $ be an $m$-manifold with a warped cylindrical end such that the warping functions $f_1,f_2$ satisfy $\int _1^{\infty}\frac {dt}{f_1(t)} = \infty$ and $\int _1^{\infty}\frac {dt}{f_2(t)} <\infty.$  Let $I = (1, \infty)\, $ and $S = S^{m-1}\, .$ Then
we have:

$(1)$ Every quasiregular
 mapping  (in particular, every $m$-harmonic morphism) from $N_1$ to $I \times _f S$
is constant, whenever $f$ satisfies $\int _1^{\infty}\frac
{dt}{f(t)} < \infty\, .$

$(2)$ Every quasiregular
 mapping  (in particular, every $m$-harmonic morphism) from $I \times _f S$ to $N_2$ is constant, whenever $f$ satisfies $\int _1^{\infty}\frac {dt}{f(t)} =
\infty\, .$

$(3)$ Every quasiregular
 mapping from $N_1$ to $N_2$ is
 constant. In particular, there is no nonconstant $m$-harmonic morphism from $N_1$ to $N_2\, .$
\end{corollary}

\begin{corollary}\label{C:12.2} Let $N_i\, ,$ and $f_i$ be as in Theorem \ref{T:12.2}, for $i = 1, 2\, .$ Let $I = (1, \infty)\, $ and $S = S^{m-1}\, .$ Then we have:

$(1)$ Every $p$-harmonic morphism  from $N_1$ to $I \times _f S$
is constant, whenever $f$ satisfies $\int
_1^{\infty}f(t)^{\frac{m-1}{1-p}} dt < \infty\, . $

$(2)$ Every $p$-harmonic morphism from $I \times _f S$ to $N_2$ is
constant, whenever $f$ satisfies $\int
_1^{\infty}f(t)^{\frac{m-1}{1-p}} dt = \infty\, . $
\end{corollary}

\section{Classification of complete stable minimal surfaces in
${\bf R}\times_f S^{2}(k)\, $ }

\begin{theorem} \label{T:3.1}
Let  $M$ be a stable minimal surface of $R \times_f S\, ,$ where
$f$ is a positive $C^2$ concave function with bounded derivative $
|f'| \le {\sqrt{k}}$ on $R$. Then $M$ is totally geodesic.
Furthermore,

$(a)$ If $(\ln f)''+k/f^2=0$ on $R$, then $M$ is a plane.

$(b)$ If $\, (\ln f)''+k/f^2\ne 0$ on $R$, then one of the
following two cases occur

$(b.1)$ $M$ is a transverse submanifold which is a slice $S(t_0)$
with $f'(t_0)=0$ as a totally geodesic submanifold of $R \times_f
S$; or

$(b.2)$  $M$ is an $\mathcal H$-submanifold which is locally the
warped product $ I\times_{ f} N^{1}$ of $I$ and a geodesic $N^{1}$
of  $S$.

Moreover, if  case $(b.1)$ occurs, then $Sec(X) =
\frac{(k-f'{}^2)}{f^2}\<X,X\>, X \in \Gamma(TM).$
 \end{theorem}
To be self-contained, we provide the following complete
\begin{proof}
By virtue of the assumption $f'' \le 0$, $ |f'| \le {\sqrt{k}}$ on
$R$, and Lemma \ref{L:2.2} , $\tilde M = {\bf R} \times_f S$ is a
complete simply-connected manifold with sectional curvature
$\tilde K \ge 0\, ,$ and $Ric^{\tilde M} \ge 0\, .$

Since $M$ is a minimal surface with Guass curvature $K$ in $\tilde
M$, it follows from the Guass curvature equation that

\begin{equation}\begin{aligned}\label{4.1} 0 & \le \tilde K \\ &  = K - h_{11}h_{22} + h_{12}^2
\\& = K + h_{11}^2 + h_{12}^2
\\& = K + \frac 12 |A|^2.\end{aligned}\end{equation} Hence,
the stability inequality \r{2.20} and \r{4.1} imply that for every
$\varphi \in C_0^{\infty}(M)\, ,$

\begin{equation}\begin{aligned}\label{4.2} -2 \int _M K \phi^2 dv & \le \int _M  |A|^2 \phi^2 dv \\
& \le \int _M (Ric^{\tilde M} (\nu) + |A|^2 )\phi^2 dv \\
& \le \int _M |\nabla \phi|^2 dv,
\end{aligned}\end{equation}

(Following \cite {W2}) Firstly, we claim if  $M$ is conformally
equivalent to the plane or equivalent if $M$ is parabolic(, i.e.
there does not exist a positive superharmonic function unless it
is a constant), then $M$ is totally geodesic:  Proceed as in \cite
[p.152-153]{W2}(in which $b = |A|^2 \phi^2$, and $c_1 =1$).

For any fixed compact set $K$ in $M$, choose a sufficiently large
$r>0$ so that the ball $B_r$ of radius $r$ covers $K$ and pick
$$\vp_r=\begin{cases}
1 &\qquad \text{on}\ K\\
0 &\qquad \text{on}\ \partial B_r\\
\text{harmonic} &\qquad \text {in}\ B_r\backslash
K\,.\end{cases}$$ Set $\phi=\vp_r$ in \r{4.2}.  Since
$\Delta\vp_r=0$ in $M\backslash K$, $\vp_r=0$ on $\partial B_r$
and by divergence theorem
\begin{equation}\begin{aligned}\label{4.3}
\int_K |A|^2 \phi^2 dv&\leq c_1\int_{B_r\backslash K}|\na\vp_r|^2 dv\\
&=c_1\int_{B_r\backslash K} div(\vp_r\na\vp_r) dv\\
&=c_1\int_{\partial K}\frac{\partial \vp_r}{\partial n} ds\\
&=c_1\int_{\Sigma}\frac{\partial \vp_r}{\partial n}
ds\end{aligned}\end{equation} where $\Sigma$ is a hypersurface
between $\partial K$ and $B_r$ and $n$ is the unit outer normal
vector.

The last step follows from the harmonicity of $\vp_r$ between
$\partial K$ and $\Sigma$.  By the maximum principle
$0\leq\vp_r\leq 1$ and $\vp_r$ increases as $r$ increases.  Then
$\vp_r$ converges to a constant function $\vp_\infty\equiv 1$.
Otherwise $\vp_\infty$ would be a nonconstant positive
superharmonic function on $M$, a contradiction to the
parabolicity.

By an interior elliptic estimate \cite{GT}, $\na\vp_r\to 0$
uniformly on compact subsets of $M\backslash K$ as $r$ tends to
$\infty$. It follows from \r{4.3}  that $A\equiv 0$ and hence $M$
is totally geodesic.

Secondly, we claim if $M$ is conformally equivalent to the unit
disk $D$ endowed with the complete metric $\frac 1{f(z)^2}|dz|^2$,
or equivalent if $M$ is hyperbolic(, i.e., $M$ is not parabolic),
then $M$ is not a stable minimal submanifold of $\tilde{M}$:
Proceed as in \cite [p.154-155]{W2} in which $c_1 = c_2 =1$ (cf.
\cite{FS}). Suppose on the contrary, then by a well-known formula,
its Gaussian curvature $K$ is given by
\begin{equation}\label{4.4}K=\frac{\Delta f}f-\frac{|\na f|^2}{f^2}\end{equation}
where $\Delta$ is the Beltrami-Laplace operator on $M$.  And
\r{4.2} implies that for every $\phi\in C^\infty_0(M)$
\begin{equation}\label{4.5}-2\int_M\phi^2\left(\frac{\Delta f}f-\frac{|\na f|^2}{f^2}\right) dv\leq
\int_M|\na\phi|^2 dv\,.\end{equation} Substituting $\phi=f\vp$
into \r{4.5} for every $\vp\in C^\infty_0(M)$, we have
\begin{equation}\label{4.6}-2\int_M f\vp^2\Delta f-|\na f|^2\vp^2 dv\leq
\int_M|\na f|^2\vp^2+|\na\vp|^2f^2+2f\vp\na f\na\vp
dv\,.\end{equation} Integration by parts and Cauchy-Schwarz
inequality yield \begin{equation}\begin{aligned}\label{4.7}
3\int_M|\na f|^2\vp^2 dv & \leq \int_M|\na\vp|^2f^2 dv -
\int_M2f\vp\na f\na\vp dv\\
& \leq 2\int_M|\na\vp|^2f^2 dv+ \int_M\vp^2|\na f|^2
dv\,.\end{aligned}\end{equation} Then \r{4.7} implies
\begin{equation}\label{4.8}\int_M|\na f|^2\vp^2 dv\leq \int_Mf^2|\na\vp|^2 dv\,.\end{equation}
Choose a cut-off function $\vp$ with compact support in $B_r$ with
$|\na\vp|\leq\frac{c}r$ and $\vp\equiv 1$ on $B_{\frac r2}$ and
 use the fact that $dv=\frac 1{f^2}dxdy$
\begin{equation}\begin{aligned}\label{4.9}
\int_{B_{\frac r2}}|\na f|^2 dv&\leq\frac{c^2}{r^2}\int_Mf^2 dv\\
&=\frac{c^2}{r^2}\int_Ddxdy\\
&=\frac{c^2\pi}{r^2}\end{aligned}\end{equation} which tends to $0$
as $r$ tends to $\infty$.  Hence $f$ is a constant on $M$,
contradicting the completeness of the metric $\frac
1{f(z)^2}|dz|^2$.

It follows from the Uniformization Theorem that $M$ is totally
geodesic. In view of Proposition \ref{P:2.1} that the assertion
($b$) follows if $(\ln f)''+k/f^2\ne 0$ on $R$; and $M$ is a
totally geodesic submanifold of nonnegative constant curvature
$(k-f'{}^2)/f^2$, if $(\ln f)''+k/f^2= 0$ on $R$.  In the latter
case, $M$ has to be flat or $M$ is a plane, since a totally
geodesic submanifold $M$ of a sphere (which has positive constant
curvature) is not stable. This completes the proof.
\end{proof}

 When $f = Const\, ,$ and $k=0\, ,$ this result recaptures the following theorem
 of Colbrie-Fisher - Schoen \cite {FS},
do Carmo - Peng \cite {DP}, and Pogorelov \cite {P}:

\begin{theorem}\label{T:7.2}
Every complete stable minimal surface in $R^3$ is a plane.
 \end{theorem}
 Which implies the Classical Bernsten Theorem \cite {B}:

 \begin{theorem}\label{T:7.3}
Every entire solution to the Minimal Surface Equation
$$ div\(\frac {\na f}{\sqrt{1 + |\na f|^2}}\) = 0 $$ on $R^2$ is an affine
function.
 \end{theorem}

\section{Classification of parabolic stable minimal hypersurfaces in
${\bf R}\times_f S^{n}(k)\, $ }

Utilizing the same technique in the last section, we obtain
\begin{theorem} \label{T:3.4} Let
$M$ be a parabolic stable minimal hypersurface in warped product
Riemannian manifolds ${\bf R}\times_f S^{n}(k)\, , $ where $f$ is
as in Theorem \ref{T:3.1}. Then $M$ is totally geodesic.
Furthermore,

$(a)$ If $(\ln f)''+k/f^2=0$ on $R$,  then $M$ is a hyperplane.

$(b)$ If $\, (\ln f)''+k/f^2\ne 0$ on $R$, then  one of the
following two cases occur

$(b.1)$ $M$ a transverse submanifold which lies in a slice
$S(t_0)$ with $f'(t_0)=0$ as a totally geodesic submanifold;

$(b.2)$  $M$ is an $\mathcal H$-submanifold which is locally the
warped product $ I\times_{ f} N^{n-1}$ of $I$ and a totally
geodesic submanifold   $N^{n-1}$ of  $S$.

Furthermore, if  case $(b.1)$ occurs, then $Ric(X) =
\frac{(n-1)(k-f'{}^2)}{f^2}\<X,X\>, X \in \Gamma(TM).$
 \end{theorem}
\begin{proof}
Proceed as in the first part of the proof of Theorem \ref{T:3.1},
$M$ is totally geodesic, and Proposition \ref{P:2.1} completes the
proof.
\end{proof}

\section{$p$-hyperbolic manifolds and stable minimal hypersurfaces in
${\bf R}\times_f S^{n}(k)\, $ }

In the course of proving Theorem \ref{T:11.1}, one has shown the
case $p > 1$ for the following

\begin{proposition}\label{P:11.1}
Every complete, simply-connected, manifold with sectional
curvature bounded above by a negative constant is $p$-hyperbolic
for all $p \ge 1\, .$ In particular, every $n$-dimensional
hyperbolic space $\mathbb{H}^n$ is $p$-hyperbolic for all $p \ge
1$.
\end{proposition}
\begin{proof}
For the case $p=1$, this follows from \cite [p.139]{Tr}.
\end{proof}

Let $\mathbb{B}^n = \{(x_1, \cdots, x_n)\in \mathbb{R}^n : x_1^2 +
\cdots + x_n^2 < 1\}$ be the unit $n$-ball. Assume that the
hyperbolic space $\mathbb{H}^n$ is modeled on the Euclidean unit
$n$-ball \big($\mathbb{B}^n,
\frac{4}{(1-|\mathbf{x}|^2)^2}d\mathbf{x}^n$\big) where
$d\mathbf{x}^n$
 is Euclidean metric and $\mathbf{x} = (x_1,\cdots,x_n)$. By proposition \ref{P:11.1}, $\mathbb{H}^n$ is $p$-hyperbolic for all $p \ge1$. We
 remark that by proceeding
exactly as in the proof of \cite[Theorem 1.3]{LWe}, one can prove
that every complete manifold $M$ that is conformally equivalent to
the unit $n$-ball $\mathbb{B}^n$ cannot be stably minimally
immersed in ${\bf R}\times_f S^{n}(k)\, ,$ where $f(x)=\sqrt {k} x
+ b\, ,$ for any constants $k \ge 0$ and $b$. This is precisely
the nonexistence theorem in $R^{n+1}$\cite [Theorem 1.3] {LWe},
since by Lemma \ref{L:2.2}, and Cartan-Ambrose-Hicks Theorem, such
${\bf R}\times_f S^{n}(k)\, $ is isometric to $R^{n+1}\, .$

\noindent {\small Department of Mathematics,}

\noindent {\small  Michigan State University, }

\noindent {\small  East Lansing, Michigan 48824-1027,}

\noindent {\small  U.S.A. }

\noindent {\small {\it Email address:} bychen@math.msu.edu}
\smallskip

\noindent {\small Department of Mathematics,}

\noindent {\small  University of Oklahoma, }

\noindent {\small  Norman, Oklahoma 73019-0315,}

\noindent {\small  U.S.A. }

\noindent {\small {\it Email address:} wwei@ou.edu}

\end{document}